\documentclass[12pt,twoside,reqno,psamsfonts]{amsart}

\usepackage[OT1]{fontenc}
\usepackage{type1cm}
\usepackage{enumerate}
\usepackage{amssymb}
\usepackage[all]{xy}
\usepackage{mathrsfs}
\usepackage[dvips]{graphicx}
\usepackage{psfrag}          
\usepackage{geometry}        
\usepackage{version}         

\numberwithin{equation}{section}

\theoremstyle{plain}
\newtheorem{theorem}{Theorem}[section]
\newtheorem{corollary}[theorem]{Corollary}
\newtheorem{proposition}[theorem]{Proposition}
\newtheorem{lemma}[theorem]{Lemma}

\theoremstyle{remark}
\newtheorem{remark}[theorem]{Remark}

\newtheorem*{ack}{Acknowledgements}

\theoremstyle{definition}
\newtheorem{definition}[theorem]{Definition}
\newtheorem{question}[theorem]{Question}

\newcommand{\sF}{\mathscr{F}}
\newcommand{\sE}{\mathscr{E}}

\newcommand{\R}{\mathbb{R}}

\newcommand{\Z}{\mathbb{Z}}
\newcommand{\N}{\mathbb{N}}

\DeclareMathOperator{\diam}{diam}

\DeclareMathOperator{\spt}{spt}

\DeclareMathOperator{\Geo}{Geo}
\DeclareMathOperator{\GeoOpt}{GeoOpt}

\def\Xint#1{\mathchoice
{\XXint\displaystyle\textstyle{#1}}%
{\XXint\textstyle\scriptstyle{#1}}%
{\XXint\scriptstyle\scriptscriptstyle{#1}}%
{\XXint\scriptscriptstyle\scriptscriptstyle{#1}}%
\!\int}
\def\XXint#1#2#3{{\setbox0=\hbox{$#1{#2#3}{\int}$ }
\vcenter{\hbox{$#2#3$ }}\kern-.6\wd0}}

\def\dashint{\Xint-}

\begin{document}

\title[Interpolated measures in $CD(K,N)$ spaces of Sturm]
{Interpolated measures with bounded density
in metric spaces satisfying the curvature-dimension conditions of Sturm}

\author{Tapio Rajala}
\address{Scuola Normale Superiore\\
Piazza dei Cavalieri 7\\
I-56127 Pisa\\ Italy}
\email{tapio.rajala@sns.it}

\thanks{The author acknowledges the support of the European Project ERC AdG *GeMeThNES* and the Academy of Finland project no. 137528.}
\date{\today}


\begin{abstract}
We construct geodesics in the Wasserstein space of probability measure
along which all the measures have an upper bound on their density
that is determined by the densities of the endpoints of the geodesic.
Using these geodesics we show that a local Poincar\'e inequality
and the measure contraction property follow from the Ricci curvature bounds defined by Sturm.
We also show for a large class of convex functionals that a local Poincar\'e inequality is implied by the 
weak displacement convexity of the functional.
\end{abstract}


\maketitle

\tableofcontents

\newpage

\section{Introduction}

A definition for lower Ricci curvature bounds in metric measure spaces using mass transportation was given independently
by Sturm \cite{S2006I, S2006II} and by Lott and Villani \cite{LV2009}. 
Both definitions use convexity inequalities for functionals in the space of probability measures. Because Sturm's
definition requires these inequalities for much smaller class of functionals it is at least a priori weaker.
In their seminal papers Sturm, and Lott and
Villani established most of the essential properties of metric measure spaces with lower Ricci curvature bounds; such
as the stability under the measured Gromov-Hausdorff convergence and coincidence with the Riemannian definition.
However, one of the basic tools for doing analysis in these spaces was still missing, namely the local
Poincar\'e inequality. 

The validity of the local Poincar\'e inequality was proved by Lott and Villani \cite{LV2007}
in the special case where the space was assumed to be nonbranching, see also \cite{vR2008}.
Very recently this nonbranching assumption was removed in \cite{R2011} for spaces with lower Ricci curvature
bounds in the sense of Lott and Villani. In this paper we go one step further and 
prove the local Poincar\'e inequalities
in the case of possibly branching spaces with Ricci curvature bounded from below in the sense of Sturm.
See Section \ref{sec:definitions} for the definitions of the Ricci curvature bounds.
The constants in the Poincar\'e inequalities we prove here are essentially the same that were obtained in \cite{R2011}.
Notice that in \cite[Theorem 2]{R2011} there should also be the $\cosh$-term in the constant which we have in
the theorem below.

\begin{theorem}\label{thm:main}
 Any $CD(K,N)$ space in the sense of Sturm, with $K \in \R$ and $N \in (1,\infty)$, supports the following weak local Poincar\'e inequality.
 For every point $x \in X$ and radius $r>0$, for any continuous function $u \colon X \to \R$ and any upper gradient $g$ of $u$
 we have
 \[
  \dashint_{B(x,r)}|u - \langle u\rangle_{B(x,r)}|dm \le C r \dashint_{B(x,2 r)}gdm,
 \]
 where the constant $C$ depends on $N$, $K$ and $r$ and can be estimated from above by
 \[
  C \le 2^{N+3}e^{\sqrt{(N-1)K^-}2r}\cosh\left(2r \sqrt{\frac{K^-}{N-1}}\right)^{N-1}
 \]
 with $K^- = \max\{-K,0\}$.
 In particular, in any $CD(0,N)$ space with $N \in (1,\infty)$ we have the uniform weak local Poincar\'e inequality
 \[
  \dashint_{B(x,r)}|u - \langle u\rangle_{B(x,r)}|dm \le 2^{N+3} r \dashint_{B(x,2 r)}gdm.
 \]
\end{theorem}
 The abbreviations for average integrals in the theorem are defined for any $A \subset X$ with $0<m(A) < \infty$ as
 \[
  \langle u \rangle_{A} = \dashint_{A}u dm = \frac1{m(A)} \int_{A} udm.
 \]
In \cite{R2011} a local Poincar\'e type inequality was also proved in $CD(K,\infty)$ spaces in the sense of 
Lott and Villani. We also have this result using the definition of Sturm.

\begin{theorem}\label{thm:main2}
 Suppose that $(X,d,m)$ is a $CD(K,\infty)$ space in the sense of Sturm.
 Then we have the local Poincar\'e type inequality
 \[
  \int_{B(x,r)}|u - \langle u\rangle_{B(x,r)}|dm \le 8 r e^{K^-r^2/3}\int_{B(x,2 r)}gdm.
 \]
\end{theorem}

The proof of the local Poincar\'e inequalities is based on the existence of geodesics in the 
Wasserstein space along which the densities of all the measures are suitably bounded from above. The existence of such
geodesic is already interesting by itself. Density bounds along geodesics have turned out to be useful 
in analysis on metric spaces and in particular on genuinely infinite
dimensional metric spaces where the lack of doubling measures restricts the use of local Poincar\'e inequalities.
See \cite{AGS2011} for recent development in this direction.
Using the notation which will be introduced in Section \ref{sec:definitions} we can state the existence
of the good geodesics as follows.

\begin{theorem}\label{thm:goodgeodesic}
 Let $(X,d,m)$ be a $CD(K,N)$ space in the sense of Sturm for some $K \in \R$ and $N \in (1,\infty]$.
 Then for any $\mu_0, \mu_1 \in \mathcal{P}^{ac}(X,m)$ with $D = \diam(\spt\mu_0\cup\spt \mu_1) < \infty$ there
 exists a geodesic $\Gamma \in \Geo(\mathcal{P}(X))$ so that $\Gamma_0 = \mu_0$, $\Gamma_1 = \mu_1$ and 
 for all $t \in [0,1]$ we have $\Gamma_t = \rho_tm$ with
 \begin{equation}\label{eq:densitybound_finite}
  ||\rho_t||_{L^{\infty}(X,m)} \le e^{\sqrt{(N-1)K^-}D} \max\{||\rho_0||_{L^{\infty}(X,m)},||\rho_1||_{L^{\infty}(X,m)}\}
 \end{equation}
 if $N < \infty$ and with
 \begin{equation}\label{eq:densitybound_infinite}
  ||\rho_t||_{L^{\infty}(X,m)} \le e^{K^-D^2/12} \max\{||\rho_0||_{L^{\infty}(X,m)},||\rho_1||_{L^{\infty}(X,m)}\}
 \end{equation}
 if $N = \infty$.
\end{theorem}

In Theorem \ref{thm:goodgeodesic} we have the existence of a good geodesic between two absolutely continuous measures. 
If in the case $N < \infty$ we
construct a similar geodesic between a point mass and an absolutely continuous measure, we obtain the so called measure
contraction property as defined by Ohta \cite{O2007}. Measure contraction property can also be regarded as a generalization of Ricci curvature bounds.
See Section \ref{sec:definitions} for the definition of this property.

\begin{theorem}\label{thm:mcp}
 Any $CD(K,N)$ space has the $MCP(K,N)$ property.
\end{theorem}

The measure contraction property, just like the local Poincar\'e inequalities, was already known to hold
in $CD(K,N)$ spaces under the nonbranching assumption \cite{S2006II}.
There are many definitions of the measure contraction
property. A stronger version than what we consider here was given by Sturm in \cite{S2006II} where he also showed that a
different type of Poincar\'e inequality follows from this definition without any assumption on nonbranching.
It should be emphasized that we prove Theorem \ref{thm:mcp} only with the weaker measure contraction property defined by Ohta.
 Using the results of this paper the implications between different conditions 
without any assumption on nonbranching can now be written as follows (compare this to the similar diagram in \cite{R2011}).

\[
\xymatrix {
   *+[F-,] {\txt{$CD(K,N)$ \`a la \\ \textsc{Lott-Villani}}} \ar @{=>} [r]
   & *+[F-,]{\txt{$CD(K,N)$ \`a la \\ \textsc{Sturm}}} \ar @{=>} [d] \ar @{=>} [r]
   & *+[F-,]{\txt{$MCP(K,N)$ \`a la \\ \textsc{Ohta}}}\\
   & *+[F]{\txt{local Poincar\'e\\ inequality}}
   &
}
\]

It is known that the measure contraction property does not imply the curvature-dimension bounds in the sense of Sturm.
Obviously a local Poincar\'e inequality does not imply $MCP(K,N)$.
So, in the above diagram the validity of only two possible implications is open:

\begin{question}\label{q1}
 Are $CD(K,N)$ spaces in the sense of Sturm also $CD(K,N)$ spaces in the sense of Lott and Villani?
\end{question}

Again this is known to be true under the extra assumption of nonbranching \cite{V2009}.
If the answer to Question \ref{q1} is positive, the methods of this paper might help in proving it.
Indeed, what would be needed for the proof would be geodesics along which all the functionals
used in the definition by Lott and Villani satisfy a convexity inequality. The theme of this paper
is to find better geodesics than the ones given directly by the curvature-dimension condition.
However, we were not able to show the existence of geodesics that would answer Question \ref{q1}.

\begin{question}\label{q2}
 Does a local Poincar\'e inequality follow already from the $MCP(K,N)$ as defined by Ohta?
\end{question}

Because the definition of the measure contraction property involves a point mass, see Section \ref{sec:definitions}, the
proof for the local Poincar\'e inequalities as introduced in \cite{R2011} does not seem to work in spaces
with only the property $MCP(K,N)$. Notice that again in nonbranching spaces $MCP(K,N)$ in the sense of Ohta
implies a local Poincar\'e inequality \cite{vR2008}.


The paper is organized as follows. In Section \ref{sec:definitions} we give the relevant definitions and background.
In Section \ref{sec:construction} we construct the good geodesics of Theorem \ref{thm:goodgeodesic}. Here the construction
in the case $N= \infty$ requires more work than in the boundedly compact case because the existence of suitable minimizers
does not immediately follow from a compactness result.

In Section \ref{sec:poincare} we derive the local Poincar\'e inequalities of Theorem \ref{thm:main} and
Theorem \ref{thm:main2} from the existence of good geodesic stated in Theorem \ref{thm:goodgeodesic}.
The validity of the local Poincar\'e inequalities are stated in a more general form in Theorem \ref{thm:general}.
In this section we also show that Theorem \ref{thm:general} can be applied for example in boundedly
compact spaces where a suitable functional is weakly displacement convex. 

In the final section, Section \ref{sec:mcp}, we prove Theorem \ref{thm:mcp} which says that the $CD(K,N)$ spaces
satisfy $MCP(K,N)$.
Here the difference to the rest of the paper is that we will need to construct the good geodesics between a point mass
and an absolutely continuous measure. However, the strategy of constructing geodesics which is used in Section 
\ref{sec:construction} works also in this case with only minor modifications.

\begin{ack}
 Many thanks are due to Luigi Ambrosio for his mentoring and valuable feedback. Special thanks are also due to
 Karl-Theodor Sturm for suggesting the approach of constructing geodesics by minimizing functionals. We also thank
 Nicola Gigli for his excellent suggestions and comments on this work.
\end{ack}

\section{Preliminaries}\label{sec:definitions}

All the metric measure spaces $(X,d,m)$ that we consider are assumed to be complete, separable and geodesic.
Recall that a metric space $(X,d)$ is called locally compact if every point has a compact neighbourhood and it is
called boundedly compact if every bounded closed set is compact. Analogously the measure $m$ is called
locally finite if every point has a neighbourhood with finite $m$-measure and it is called boundedly
finite if every bounded set has finite $m$-measure.
Notice that locally finite measures in boundedly compact spaces are also boundedly finite.

We denote the support of a measure $\mu$ by $\spt \mu$.
By $\mathcal{P}(X)$ we mean the set of Borel probability measures on $X$ and by $\mathcal{P}^{ac}(X,m) \subset \mathcal{P}(X)$
the set of probability measures in $X$ that are absolutely continuous with respect to the measure $m$.
When we say that a measure is absolutely continuous without reference to any other measure it is understood that 
it is absolutely continuous with respect to the fixed measure $m$ of the metric measure space.
We say that a measure $m$ is doubling (with a constant $1\le C<\infty$) if for all $x \in X$ and $0 < r < \diam(X)$ we have
\[
 m(B(x,2r)) \le C m(B(x,r)).
\]

Recall that any geodesic in a metric space $(X,d)$ can be reparametrized to be a continuous mapping $\gamma \colon [0,1] \to X$ with
\[
 d(\gamma(t),\gamma(s)) = |t-s|d(\gamma(0),\gamma(1)) \qquad \text{for all } 0 \le t \le s \le 1.
\]
We denote the space of all the geodesics of the space $X$ with such parametrization by $\Geo(X)$.
For a geodesic $\gamma \in \Geo(X)$ and $t \in [0,1]$ we will use the abbreviation $\gamma_t = \gamma(t)$.
We equip the space $\Geo(X)$ with a distance 
\[
 d^*(\gamma,\tilde\gamma) = \max_{t \in [0,1]}d(\gamma_t,\tilde\gamma_t).
\]

A metric space is called branching if there exist two distinct geodesics starting from the same point
which follow the same path for some initial time interval and then become disjoint. A nonbranching space
is a space that is not branching.

\subsection{Metric spaces with a local Poincar\'e inequality}
The importance of doubling measures and local Poincar\'e inequalities in the analysis on metric spaces
became evident in the pioneering works of Cheeger \cite{C1999} and Heinonen and Koskela \cite{HK1998}.
Since then these two properties have become standard assumptions in the theory.
Studying which spaces satisfy them is now a significant part of the theory. For an introduction on the
analysis done on metric spaces we refer to the book by Heinonen \cite{H2001}.

A metric measure space $(X,d,m)$ admits a weak local $(q,p)$-Poincar\'e inequality with $1 \le p \le q < \infty$ if there
exist constants $\lambda \ge 1$ and $ 0 < C < \infty$ such that for any continuous function $u$ defined on $X$, any point 
$x \in X$ and radius $r>0$ such that $m(B(x,r)) > 0$ and any upper gradient $g$ of $u$ we have
\begin{equation}\label{eq:Poincaredefinition}
  \left(\dashint_{B(x,r)}|u - \langle u\rangle_{B(x,r)}|^qdm\right)^{1/q} \le C r \left(\dashint_{B(x,\lambda r)}g^pdm\right)^{1/p},
\end{equation}
where the barred integral denotes the average integral and $\langle u\rangle_{B(x,r)}$ denotes the average of $u$ in the
ball $B(x,r)$. Recall that, as introduced in \cite{HK1998}, a Borel function $g \colon X \to [0,\infty]$ is an upper gradient
of $u$ if for any $\gamma \in \Geo(X)$ with length $l(\gamma) < \infty$ we have
\[
 |u(\gamma_0) - u(\gamma_1)| \le l(\gamma) \int_0^1g(\gamma_t)dt.
\]

We will consider here weak local $(1,1)$-Poincar\'e inequalities which we simply call weak local Poincar\'e inequalities.
The word \emph{weak} here refers to the fact that we allow the ball on the right-hand side of \eqref{eq:Poincaredefinition} to be
larger than the one on the left. If the balls on both sides of the inequality can be taken to be the same, meaning that we can
take $\lambda = 1$, the inequality is called a strong local Poincar\'e inequality. In a doubling geodesic metric space the weak
local Poincar\'e inequality implies the strong one, with possibly a different constant $C$, see \cite{HK1995} and also \cite{HK2000}.

We already know from a result proved by Buser \cite{B1982} that a Riemannian manifold with nonnegative Ricci curvature supports a
local Poincar\'e inequality. Moreover, in the case of measured Gromov-Hausdorff limits of Riemannian manifolds with Ricci curvature
bounded below local Poincar\'e inequalities are also known to hold \cite{CC2000}. In \cite{LV2007} a local Poincar\'e inequality
was proved in nonbranching metric spaces with nonnegative Ricci curvature, see also \cite{vR2008}. In \cite{R2011} this result was generalized (with the 
definition used by Lott and Villani) by removing the assumption for the space to be nonbranching.
This paper continues this line of investigation.
Notice that Poincar\'e inequalities have also been proved in many other classes of metric spaces, for example in locally linearly 
contractible Ahlfors-regular metric spaces \cite{S1996}.

\subsection{Optimal mass transportation and the Wasserstein distance}

The definitions of lower Ricci curvature bounds considered by Lott, Sturm and Villani use the theory of optimal mass transportation.
This theory has a long history starting from the work of Monge in the 18th century \cite{M1781}. In the modern formulation of the
mass transportation problem, which was developed by Kantorovich \cite{K1942, K1948}, the transportation of the mass is optimized among all
possible measures with correctly fixed marginals.
The main reason for using measures instead of mappings for transporting mass is that with measures in most situations we have the existence
of optimal transportation, and more importantly the existence of a transport to begin with. The problem with transport maps is that
they cannot split measure, which is sometimes necessary. See for instance the recent paper \cite{G2011} for the assumptions that are
needed for the existence of optimal mappings in the spaces we study here.
We refer to the book by Villani \cite{V2009} for a detailed account on the history and modern theory of optimal mass transportation.

The consideration of optimal mass transportation leads to the definition of Wasserstein space $(\mathcal{P}(X), W_2)$.
The distance between two probability measures $\mu, \nu \in \mathcal{P}(X)$ in this space is given by
\[
 W_2(\mu, \nu) = \left(\inf\left\{\int_{X\times X} d(x,y)^2d\sigma(x,y)\right\}  \right)^{1/2},
\]
where the infimum is taken over all $\sigma \in \mathcal{P}(X \times X)$ with $\mu$ as its first marginal and $\nu$ as the second,
i.e. $\mu(A) = \sigma(A \times X)$ and $\nu(A) = \sigma(X \times A)$ for all Borel subsets $A$ of the space $X$.
Notice that in the case where the distance $d$ is not bounded the function $W_2$ is strictly speaking not a distance as the
above infimum can also take an infinite value. This will not be an issue for us since all the measures in $\mathcal{P}(X)$ we
consider have bounded support.

An important thing for us to notice is that any geodesic $\Gamma \in \Geo(\mathcal{P}(X))$ between two measures $\mu, \nu \in \mathcal{P}(X)$
in the space $(\mathcal{P}(X), W_2)$  can be realized as a measure $\pi \in \mathcal{P}(\Geo(X))$ so that $\Gamma_t = (e_t)_\#\pi$, where
$e_t(\gamma) = \gamma_t$ for any geodesic $\gamma$ and $t \in [0,1]$ and $f_\#\mu$ denotes the push-forward of the measure $\mu$
under $f$, see for example \cite[Corollary 7.22]{V2009}. This realization is convenient for us when we want to
translate information from the geodesics on $\mathcal{P}(X)$ to the geodesics on $X$.
The space consisting of all measures $\pi\in \mathcal{P}(\Geo(X))$
for which the mapping $t \mapsto (e_t)_\#\pi$ is a geodesic in $\mathcal{P}(X)$ from $\mu = (e_0)_\#\pi$ to $\nu = (e_1)_\#\pi$ 
is denoted by $\GeoOpt(\mu, \nu)$. We equip this space with a distance
\[
 \mathcal{W}_2(\pi_1,\pi_2) = \sup_{t \in [0,1]} W_2((e_t)_\#\pi_1,(e_t)_\#\pi_2).
\]

\subsection{Ricci curvature lower bounds in metric spaces}

There are three different sets of definitions of lower Ricci curvature bounds in metric spaces that are discussed in this paper.
In two of the definitions suitable functionals are required to satisfy a convexity inequality between measures in the Wasserstein space
$(\mathcal{P}(X),W_2)$.

One set of definitions follows the approach by Lott and Villani \cite{LV2009} and requires that 
between any two probability measures that have bounded Wasserstein distance between them
there is at least one geodesic in the Wasserstein space along which all the 
functionals in a displacement convexity class $\mathcal{DC}_N$ satisfy a convexity inequality.
The second set of definitions, following the work of Sturm \cite{S2006I, S2006II}, requires the same
convexity inequality only for the critical entropy functionals.
The third definition, the measure contraction property, is different from the two previous ones. It requires
the existence of a geodesic between any point mass and any uniformly distributed measure along which we have
a suitable density bound.

To be more precise on the first two sets of definitions, in their paper Lott and Villani gave a definition 
for nonnegative $N$-Ricci curvature with
$N \in [1,\infty)$ and a definition for $\infty$-Ricci curvature being bounded below by $K \in \R$.
Sturm on the other hand defined for all $N \in [1,\infty]$ spaces where $N$-Ricci curvature is bounded from
below by a constant $K \in \R$.
Although Sturm's definition is a priori weaker, in nonbranching metric spaces
these two sets of definitions agree, see for example \cite{V2009}. In nonbranching spaces both these definitions,
with $N < \infty$, also imply the measure contraction property.

Let us now define for $N \in (1,\infty)$ the spaces where $N$-Ricci curvature is bounded from 
below by a constant $K \in \R$ in the sense of Sturm.
For this we will need the R\'enyi entropy functional $\sE_N \colon \mathcal{P}(X) \to [-\infty,0]$ defined as
\[
 \sE_N(\mu) = -\int_X \rho^{1-1/N} dm,
\]
where $\mu = \rho m + \mu^s$ with $\mu^s \perp m$.

For $K \in \R$ and $N \in (1,\infty)$, we define
\[
 \beta_t(x,y) = \begin{cases}
                 \infty  & \text{if }K> 0 \text{ and }\alpha >\pi,\\
                 \left(\frac{\sin(t\alpha)}{t\sin \alpha}\right)^{N-1}  & \text{if }K> 0 \text{ and }\alpha \in [0,\pi],\\
                 1  & \text{if }K = 0,\\
                 \left(\frac{\sinh(t\alpha)}{t\sinh \alpha}\right)^{N-1}  & \text{if }K < 0,
                \end{cases}
\]
where
\[
 \alpha = \sqrt{\frac{|K|}{N-1}}d(x,y).
\]
Sometimes we write $\beta_t(l)$ which is understood to be the above quantity $\beta_t(x,y)$ with $d(x,y)$ replaced by $l$.

\begin{definition}
We say that a locally compact $\sigma$-finite metric measure space $(X,d,m)$ is a $CD(K,N)$ space (in the sense of Sturm), with the interpretation that it has $N$-Ricci curvature
bounded below by $K$, if for any two measures $\mu_0, \mu_1 \in \mathcal{P}(X)$ with $W_2(\mu_0,\mu_1)<\infty$ there
exists $\pi \in \GeoOpt(\mu_0,\mu_1)$ so that along the Wasserstein geodesic $\mu_t = (e_t)_\#\pi$ for every $t \in [0,1]$
and $N' \ge N$ we have
\begin{equation}\label{eq:CD-def}
 \sE_{N'}(\mu_t) \le - \iint_{X\times X}(1-t)\left(\frac{\beta_{1-t}(x_0,x_1)}{\rho_0(x_0)}\right)^{\frac1{N'}}
                 + t\left(\frac{\beta_{t}(x_0,x_1)}{\rho_1(x_1)}\right)^{\frac1{N'}}d\sigma(x_0,x_1),
\end{equation}
where we have written $\mu_0 = \rho_0m + \mu_0^s$ and $\mu_1 = \rho_1m + \mu_1^s$ with $\mu_0^s \perp m$, $\mu_1^s \perp m$ and 
$\sigma = (e_0,e_1)_\#\pi$.
\end{definition}

In this paper we will only need the above inequality with $N' = N$. 
From the Bishop-Gromov inequality in $CD(K,N)$ spaces \cite[Theorem 2.3]{S2006II} we have the doubling property of $CD(K,N)$ spaces.
Recall the notation $K^-=\max\{-K,0\}$. 
\begin{proposition}\label{prop:doubling}
 Any $CD(K,N)$ space with diameter bounded from above by $L$ is doubling with a constant
 \[
  2^N \cosh\left(L \sqrt{\frac{K^-}{N-1}}\right)^{N-1}.
 \]
 In particular, any $CD(0,N)$ space is doubling with a constant $2^N$.
\end{proposition}

The Ricci curvature bound from below without reference to the dimension of the space is defined using the Shannon entropy 
$\sE_\infty \colon \mathcal{P}(X) \to [-\infty,\infty]$ which is defined as
\[
 \sE_\infty(\mu) = \int_X \rho\log\rho dm,               
\]
if $\mu = \rho m$ is absolutely continuous with respect to $m$ and $\infty$ otherwise.

\begin{definition}
We say that $(X,d,m)$, with a locally finite measure $m$, is a $CD(K,\infty)$ space (in the sense of Sturm), with the interpretation that it has $\infty$-Ricci curvature
bounded below by $K$, if for any two measures $\mu_0, \mu_1 \in \mathcal{P}(X)$ with $W_2(\mu_0,\mu_1)<\infty$ there
exists $\pi \in \GeoOpt(\mu_0,\mu_1)$ so that along the Wasserstein geodesic $\mu_t = (e_t)_\#\pi$ for every $t \in [0,1]$ we have
\begin{equation}\label{eq:CD-def2}
 \sE_{\infty}(\mu_t) \le (1-t)\sE_{\infty}(\mu_0) + t\sE_{\infty}(\mu_1) - \frac{K}2t(1-t)W_2^2(\mu_0,\mu_1).
\end{equation}
\end{definition}

Although $CD(K,\infty)$ spaces are not doubling, we have bounds on the volume growth of balls, see \cite[Theorem 4.24]{S2006I}.
When we combine this with the fact that $m$ is locally finite we conclude that $m$ is actually boundedly finite.

\begin{proposition}\label{prop:doundedness}
 The measure $m$ in a $CD(K,\infty)$ space $(X,d,m)$ is boundedly finite.
\end{proposition}

The third generalization of Ricci curvature bounds that we consider here is the measure contraction property, see \cite{O2007}
and also \cite{S2006II}.

\begin{definition}
A space $(X,d,m)$ is said to satisfy the measure contraction property $MCP(K,N)$ (in the sense of Ohta) if for every $x \in X$
and $A \subset X$ (and $A \subset B(x,\pi\sqrt{(N-1)/K})$ if $K > 0$) with $0 < m(A) < \infty$ there exists 
\[
 \pi \in \GeoOpt\left(\delta_x,\frac1{m(A)}m|_A\right) 
\]
so that
\begin{equation}\label{eq:MCPdef}
 dm \ge (e_t)_\#\left(t^N\beta_t(\gamma_0,\gamma_1)m(A)d\pi(\gamma)\right).
\end{equation}
\end{definition}

In the stronger definition of measure contraction property given by Sturm \cite{S2006II} the requirement for contraction is given 
globally with a collection of Markov kernels $(P_t)_{t \in (0,1)}$ from $X^2$ to $X$ so that both of the parameters of the kernel 
can be thought of as the point mass towards which we can contract. In nonbranching metric spaces the two definitions of measure 
contraction property agree because the Markov kernels in these spaces are uniquely determined by the unique geodesics between 
points (up to a set of $m\times m$-measure zero).

In the proofs we will use the following abbreviations:
\[
  C(N,K,D) = \begin{cases}
              e^{\sqrt{(N-1)K^-}D/2},& \text{if } N < \infty,\\
              e^{K^-D^2/8}, & \text{if } N= \infty
             \end{cases}
\]
and 
\[
 P(N,K,D) = \prod_{n=0}^\infty C(N,K,2^{-n}D) = \begin{cases}
             e^{\sqrt{(N-1)K^-}D},& \text{if } N < \infty,\\
             e^{K^-D^2/12}, & \text{if } N= \infty.
            \end{cases}
\]

\section{Construction of good geodesics with bounded density}\label{sec:construction}

In light of the approach taken in \cite{R2011} we know that a local Poincar\'e inequality in a $CD(K,N)$ space
will follow once we have found for any two absolutely continuous measures $\mu_0$ and $\mu_1$, with densities bounded from above,
a geodesic in the Wasserstein space between them so that every measure along the geodesic is absolutely continuous and has a 
suitable upper bound on its density. We have stated the existence of such geodesics in Theorem \ref{thm:goodgeodesic}.

In the case of $CD(K,N)$ spaces in the sense of Lott and Villani in \cite[Lemma 1]{R2011} the needed geodesics were given directly
by the curvature-dimension condition. The upper bound on the density along these geodesics was obtained in a standard way by 
taking the limit as $p \to \infty$ of the $L^p$-norms of the densities of the measures. 
This was possible because the norms to the power $p$ belong to all the displacement convexity
classes $\mathcal{DC}_N$. In $CD(K,N)$ spaces we only have the entropy functionals to work with and 
because of this we have to work a bit more to get the $L^\infty$-bound.
It is interesting to notice that in fact the existence of good geodesics and a local Poincar\'e inequality follow
already from the weak displacement convexity of any of the $L^p$-norms to the power $p$, see Theorem \ref{thm:convexitypoincare}.
Such requirement is weaker than the $CD(0,\infty)$ condition, at least in the sense of Lott and Villani.

To construct the geodesic along which we have the density bound we employ a beautiful idea suggested by K.-T. Sturm.
We first define the geodesic in the midpoint by selecting one of the good measures which belong to the set of all the
possible midpoints along geodesics between the measures $\mu_0$ and $\mu_1$. After this we define in the same manner
the midpoints between the previously selected one and the endpoints $\mu_0$ and $\mu_1$, respectively. Continuing 
this procedure inductively we define the geodesic on a dense set of parameters. A standard completion then 
gives the full geodesic.

There are two things that have to be checked in order to ensure that the measures along the geodesic indeed have the correct density bound.
Firstly, all the midpoints we have selected should have the bound. Secondly, this should imply that the bound is valid at all
measures along the geodesic. This latter point is easy to prove as it follows directly from the lower semicontinuity of suitable 
functionals in the Wasserstein space. The slightly harder part is to find the correct midpoints. 
The general scheme of selecting the midpoints, which we again learned from K.-T. Sturm, uses minimizers of suitable functionals.

The functionals which we minimize here are natural for the problem: they simply measure the excess mass of the measure above
a given density threshold. We want to show that there exists a measure among the midpoints with zero excess mass meaning that the density
of the measure is bounded from above by the threshold. To this aim we first of all prove
that there exists a minimizer of this functional.
In boundedly compact spaces this follows using the direct method in calculus of variations, because the 
functional is lower semicontinuous and the set of midpoints is compact. 
In $CD(K,\infty)$ spaces, which usually are not boundedly compact, we show by hand that there exists a sequence converging to a minimizer.
The claim is then that the functional at the minimizer is 
indeed zero. To prove this we have to use our assumption that we are in a $CD(K,N)$ space. This allows 
us to ``spread'' the excess mass (if there is any) to a larger set when measured with the underlying measure $m$.
This spreading of mass then proves that actually there can be no excess mass at all at the minimum.
Hence the upper bound on the density and the local Poincar\'e inequality follow.

We now gather all the parts that are needed for the proof. The role of each part should be clear from the outline we gave for the proof.

\subsection{Spreading mass using the curvature-dimension conditions}

The spreading of the excess mass will be done using the following proposition, which we could also derive directly from the
Brunn-Minkowski inequality \cite[Proposition 2.1]{S2006II}. As we will later note in Section \ref{sec:poincare}
such spreading can be done in many other spaces besides the $CD(K,N)$ spaces. This leads to another class of metric measure 
spaces with good geodesics and local Poincar\'e inequalities. However, we will now concentrate only on the $CD(K,N)$ spaces of Sturm.
Because any $CD(K,N)$ space is a $CD(K',N)$ space for all $K' \ge K$ in this section it always suffices to consider only the case $K \le 0$.

\begin{proposition}\label{prop:spreading}
 Suppose that $(X,d,m)$ is a $CD(K,N)$ space with $K \in \R$ and $N \in (1,\infty]$.
 Then for any $\mu_0, \mu_1 \in \mathcal{P}^{ac}(X,m)$ with bounded support and with densities $\rho_0$ and $\rho_1$
 bounded from above there exists $\pi \in \GeoOpt(\mu_0,\mu_1)$ so that
 \begin{equation}\label{eq:bigsupport}
  m(\{x \in X ~:~ \rho_{\frac12}(x) > 0\}) \ge 
  \frac1{C(N,K,D) \max\{||\rho_0||_{L^{\infty}(X,m)},||\rho_1||_{L^{\infty}(X,m)}\}},
 \end{equation}
 where $(e_\frac12)_\#\pi = \rho_\frac12 m + \mu_\frac12^s$ with $\mu_\frac12^s \perp m$ and 
 $D$ is an upper bound for the length of $\pi$-almost every $\gamma \in \Geo(X)$.
\end{proposition}

\begin{proof}
  Write
  \[
   M = \max\{||\rho_0||_{L^{\infty}(X,m)},||\rho_1||_{L^{\infty}(X,m)}\}
  \]
  and
  \[
   E = \{x \in X ~:~ \rho_{\frac12}(x) > 0\}.
  \]

  Let us first prove the claim for $N < \infty$. Let $\pi \in \GeoOpt(\mu_0,\mu_1)$ be a measure satisfying \eqref{eq:CD-def}
  which is concentrated on geodesics with length at most $D$. From \eqref{eq:CD-def} we get
 \begin{align*}
  \sE_N((e_{\frac12})_\#\pi) 
   & \le - \frac12\iint_{X\times X}\left(\frac{\beta_{\frac12}(x_0,x_1)}{\rho_0(x_0)}\right)^{\frac1N}
                                   +\left(\frac{\beta_{\frac12}(x_0,x_1)}{\rho_1(x_1)}\right)^{\frac1N}d\sigma(x_0,x_1)\\
   & \le - \left(e^{\sqrt{(N-1)K^-}D/2}M\right)^{-\frac1N},
 \end{align*}
 because for $K \le 0$ we have
 \begin{align*}
   \beta_{\frac12}(x_0,x_1) & = \left(\frac{\sinh(\frac\alpha2)}{\frac12\sinh \alpha}\right)^{N-1}
                       = \left(\frac{2}{e^{\frac\alpha2} + e^{-\frac\alpha2}}\right)^{N-1} \ge e^{-\frac\alpha2(N-1)}\\
                        &  \ge \exp\left(-\sqrt{\frac{|K|}{N-1}}\frac{D}2(N-1)\right) = e^{-\sqrt{(N-1)|K|}D/2}.
 \end{align*} 
 On the other hand by Jensen's inequality we have
 \[
  \sE_N((e_{\frac12})_\#\pi) = - \int_E \rho_{\frac12}^{1-\frac1N}dm 
                               \ge -m(E)\left(\frac1{m(E)}\int_E \rho_{\frac12} dm\right)^{1-\frac1N}
                                \ge - m(E)^{\frac1N}.
 \]
 Combination of these two inequalities gives \eqref{eq:bigsupport}.

 Let us then prove the case $N = \infty$.
 Let $\pi \in \GeoOpt(\mu_0,\mu_1)$ be a measure satisfying \eqref{eq:CD-def2}
 which is concentrated on geodesics with length at most $D$.
 From \eqref{eq:CD-def2} we get
 \[
  \sE_\infty((e_{\frac12})_\#\pi) \le \frac12\sE_\infty(\mu_0) + \frac12\sE_\infty(\mu_1) + \frac{K^-}2\frac12\left(1-\frac12\right)D^2
                                  \le \log M + \frac{K^-D^2}8.
 \]
  Again, using Jensen's inequality we get
 \[
  \sE_\infty((e_{\frac12})_\#\pi) = \int_E \rho_{\frac12} \log \rho_{\frac12} dm \ge \log \frac1{m(E)}
 \]
 and the combination of these two estimates gives the claim.
\end{proof}

\subsection{The set of intermediate points}

We define for any two measures $\mu_0, \mu_1 \in \mathcal{P}(X)$ with $W_2(\mu_0,\mu_1) < \infty$ the set
of all the intermediate points (with a parameter $\lambda \in (0,1)$) 
 as
\begin{align*}
\mathcal{I}_\lambda(\mu_0,\mu_1) = \{\nu \in \mathcal{P}(X) ~:~ & W_2(\mu_0,\nu) = \lambda W_2(\mu_0,\mu_1) \text{ and }\\
                                                                     & W_2(\mu_1,\nu) = (1-\lambda) W_2(\mu_0,\mu_1)\}. 
\end{align*}
In the case $\lambda = \frac12$ we call the set of intermediate points the set of midpoints and write
\[
 \mathcal{M}(\mu_0,\mu_1) = \mathcal{I}_\frac12(\mu_0,\mu_1).
\]
For all the results in this paper except the measure contraction property it is enough to consider the set of midpoints.

We will use compactness of $\mathcal{I}_\lambda(\mu_0,\mu_1)$ to find the minimizers if the space
$(X,d)$ is boundedly compact. First step in this direction is to show that in general the set
 $\mathcal{I}_\lambda(\mu_0,\mu_1)$ is at least closed in $(\mathcal{P}(X),W_2)$. This fact will also be
needed in the $CD(K,\infty)$ spaces.

\begin{lemma}\label{lma:bdd}
 Assume that $(X,d)$ is a metric space and that $\mu_0, \mu_1 \in \mathcal{P}(X)$ have bounded support.
 Then for all $\lambda \in (0,1)$ the set $\mathcal{I}_\lambda(\mu_0,\mu_1)$ is closed in $(\mathcal{P}(X),W_2)$.
\end{lemma}
\begin{proof}
 Take any sequence
 $(\nu_n)_{n=1}^\infty \subset \mathcal{I}_\lambda(\mu_0,\mu_1)$ such that
 \[
  \nu_n \to \nu \in \mathcal{P}(X)\qquad \text{in the }W_2\text{-distance as }n \to \infty.
 \]
 Then
 \[
  \max\{|W_2(\mu_0,\nu) -W_2(\mu_0, \nu_n)|,|W_2(\mu_1,\nu) -W_2(\mu_1, \nu_n)|\} \le W_2(\nu, \nu_n) \to 0
 \]
 as $n \to \infty$. So, 
 \[
  W_2(\mu_0,\nu) = \lambda  W_2(\mu_0,\mu_1) \qquad \text{and} \qquad W_2(\mu_1,\nu) = (1-\lambda) W_2(\mu_0,\mu_1)  
 \]
 and thus $\nu \in \mathcal{I}_\lambda(\mu_0,\mu_1)$. 
\end{proof}

To get the compactness of $\mathcal{I}_\lambda(\mu_0,\mu_1)$ we need to assume that the space is boundedly compact.

\begin{lemma}\label{lma:cmpt}
 Assume that $(X,d)$ is a boundedly compact metric space and that $\mu_0, \mu_1 \in \mathcal{P}(X)$ have bounded support.
 Then for all $\lambda \in (0,1)$ the set $\mathcal{I}_\lambda(\mu_0,\mu_1)$ is compact in $(\mathcal{P}(X),W_2)$.
\end{lemma}
\begin{proof}
 Because the measures $\mu_0$ and $\mu_1$ have bounded support and $(X,d)$ is boundedly compact,
 we can cover the set $\mathcal{I}_\lambda(\mu_0,\mu_1)$ with a finite number of
 balls with arbitrarily small radius. Therefore $\mathcal{I}_\lambda(\mu_0,\mu_1)$ is relatively 
 compact in $\mathcal{P}(X)$ and hence by Lemma \ref{lma:bdd} it is compact.
\end{proof}

 An easy consequence of the compactness of the set of intermediate points is the compactness of geodesics between the corresponding
 measures. This will be used in the proof of the measure contraction property. Recall that in Section \ref{sec:definitions} we 
 defined the distance $\mathcal{W}_2$ in the space $\mathcal{P}(\Geo(X))$ as
 \[
  \mathcal{W}_2(\pi_1,\pi_2) = \sup_{t \in [0,1]} W_2((e_t)_\#\pi_1,(e_t)_\#\pi_2).
 \]

\begin{lemma}\label{lma:cmpt_geod}
 Assume that $(X,d)$ is a boundedly compact metric space and that $\mu_0, \mu_1 \in \mathcal{P}(X)$ have bounded support.
 Then the set $\GeoOpt(\mu_0,\mu_1)$ is compact in the space $(\mathcal{P}(\Geo(X)),\mathcal{W}_2)$.
\end{lemma}
\begin{proof}
 Let $(\pi_n)_{n=1}^\infty$ be a sequence in $\GeoOpt(\mu_0,\mu_1)$. Then by Lemma \ref{lma:cmpt} there exists a subsequence
 (which we still write as $(\pi_n)_{n=1}^\infty$) for which $((e_\frac12)_\#\pi_n)_{n=1}^\infty$ converges to a measure in
 $\mathcal{M}(\mu_0,\mu_1)$. Going into a further subsequence gives the convergence of also $((e_\frac14)_\#\pi_n)_{n=1}^\infty$
 and $((e_\frac34)_\#\pi_n)_{n=1}^\infty$ to measures in $\mathcal{I}_\frac14(\mu_0,\mu_1)$ and 
 $\mathcal{I}_\frac34(\mu_0,\mu_1)$ respectively. Taking further subsequences and finally a diagonal sequence gives convergence
 of $((e_\lambda)_\#\pi_n)_{n=1}^\infty$ for a dense set of parameters $\lambda \in [0,1]$. This gives a measure
 $\pi \in \GeoOpt(\mu_0,\mu_1)$ to which $(\pi_n)_{n=1}^\infty$ converges in the $\mathcal{W}_2$-distance.
\end{proof}

 The next lemma gives the needed convexity-type properties of the set $\mathcal{I}_\lambda(\mu_0,\mu_1)$.

\begin{lemma}\label{lma:combined}
 Suppose $\mu_0, \mu_1 \in \mathcal{P}(X)$ with $W_2(\mu_0,\mu_1)<\infty$. Then
 for any $\pi \in \GeoOpt(\mu_0,\mu_1)$ and any Borel function $f \colon \Geo(X) \to [0,1]$ with $c = (f\pi)(\Geo(X)) \in (0,1)$ we have
 \[
  (e_{\lambda})_\#\left((1-f)\pi\right) + c\nu \in \mathcal{I}_\lambda(\mu_0, \mu_1)
 \]
 with every 
 \[
  \nu \in \mathcal{I}_\lambda\left(\frac1{c} (e_{0})_\#\left(f\pi\right),
      \frac1{c} (e_{1})_\#\left(f\pi\right)\right).
 \]             
\end{lemma}
\begin{proof}
 Since $W_2^2$ is easily seen to be jointly convex, we have
 \begin{align*}
  W_2^2\big((e_{\lambda})_\#&\left((1-f)\pi\right) + c\nu, (e_{0})_\#\pi\big)\\
     = ~&  W_2^2\left((e_{\lambda})_\#\left((1-f)\pi\right) + c\nu, (e_{0})_\#\left((1-f)\pi\right) + (e_{0})_\#\left(f\pi\right)\right) \\
     \le ~& (1-c)W_2^2\left(\frac1{1-c}(e_{\lambda})_\#\left((1-f)\pi\right), \frac1{1-c}(e_{0})_\#\left((1-f)\pi\right)\right)\\
         & + cW_2^2 \left(\nu, \frac1c(e_{0})_\#\left(f\pi\right)\right) \\
     = ~& (1-c)\lambda^2W_2^2\left(\frac1{1-c}(e_{1})_\#\left((1-f)\pi\right), \frac1{1-c}(e_{0})_\#\left((1-f)\pi\right)\right)\\
         & + c\lambda^2W_2^2 \left(\frac1c(e_{1})_\#\left(f\pi\right), \frac1c(e_{0})_\#\left(f\pi\right)\right) \\
     = ~& \lambda^2W_2^2 \left(\frac1c(e_{1})_\#\pi, \frac1c(e_{0})_\#\pi\right).
 \end{align*}
 Similarly,
 \[
  W_2\left((e_{\lambda})_\#\left((1-f)\pi\right) + c\nu, (e_{1})_\#\pi\right) \le (1-\lambda)W_2 \left(\frac1c(e_{1})_\#\pi, \frac1c(e_{0})_\#\pi\right)
 \]
 and hence the claim follows.
\end{proof}
%

\subsection{The excess mass functional}

We define for all thresholds $C \ge 0$ the excess mass functional $\mathcal{F}_C \colon \mathcal{P}(X) \to [0,1]$ as
\[
 \mathcal{F}_C(\mu) = ||(\rho-C)^+||_{L^1(X,m)} + \mu^s(X),
\]
where $\mu = \rho m + \mu^s$ with $\mu^s \perp m$, and $a^+ = \max\{0,a\}$. The crucial property of this functional is that it is
lower semicontinuous in the Wasserstein space $(\mathcal{P}(X),W_2)$.

\begin{lemma}\label{lma:lsc}
 Let $(X,d)$ be a bounded metric space with a finite measure $m$.
 Then for any $C \ge 0$ the functional $\mathcal{F}_C$ is lower semicontinuous in $(\mathcal{P}(X),W_2)$.
\end{lemma}
\begin{proof}
 For locally compact spaces a proof of this fact can be found for example from \cite[Theorem 30.6]{V2009}.
 For spaces which are not locally compact the lower semicontinuity can be proved via a duality formula similar
 to \cite[Lemma 9.4.4]{AGS2008}. Namely, $\mathcal{F}_C$ can be represented as the supremum of continuous functionals:
 \begin{equation}\label{eq:Frep}
  \mathcal{F}_C(\mu) =  \sup \left\{\int_Xg(x)d\mu(x) - C \int_Xg(x)dm(x) ~:~g \in C(X),~ 0 \le g \le 1\right\}.
 \end{equation}
 Therefore it is lower semicontinuous.

 Let us verify \eqref{eq:Frep}. Inequality in one direction is obvious since
 \[
  \int_Xgd\mu - C \int_Xgdm = \int_X(\rho-C)g dm + \int_Xg d\mu^s \le \mathcal{F}_C(\mu).
 \]

 The other direction follows from the fact that the probability measures are Radon.
 Take $\epsilon > 0$. To handle the singular part of $\mu$ take compact $E_1 \subset X$ such that
 \[
  \mu^s(E_1) \ge \mu^s(X) -\epsilon \quad\text{ and }\quad m(E_1)=0.  
 \]
 Take also an open set $O_1 \subset X$ 
 with $E_1 \subset O_1$ and $m(O_1)\le\epsilon$.
 To deal with the absolutely continous part take a compact set
 \[
  E_2 \subset \{x \in X~:~\rho(x) \ge C\}
 \]
 with
 \[
  \mu(E_2) \ge \mu(\{x \in X~:~\rho(x) \ge C\}) - \epsilon,
 \]
 and an open set $O_2 \subset X$ with $E_2 \subset O_2$ and $m(O_2 \setminus E_2)\le \epsilon$.

 Now let $g \in C(X)$ be such that $0 \le g(x) \le 1$ for all $x \in X$, $g = 1$ in $E_1 \cup E_2$
 and $g = 0$ outside $O_1 \cup O_2$. Then
 \begin{align*}
  \int_Xgd\mu & - C \int_Xgdm = \int_X(\rho-C)g dm + \int_Xg d\mu^s \\
     & \ge \int_{E_1\cup E_2}(\rho-C) dm + \int_{(O_1\cup O_2) \setminus (E_1 \cup E_2)}(\rho-C)gdm + \mu^s(E_1\cup E_2) \\
     & \ge \mathcal{F}_C(\mu) -2\epsilon - Cm\left((O_1\cup O_2) \setminus (E_1 \cup E_2)\right) \ge \mathcal{F}_C(\mu) - 2(C+1)\epsilon
 \end{align*}
 proving \eqref{eq:Frep}. 
\end{proof}

Combining Lemma \ref{lma:cmpt} with Lemma \ref{lma:lsc} we get the existence of minimizers of $\mathcal{F}_C$
in $\mathcal{I}_\lambda(\mu_0,\mu_1)$ in boundedly compact metric spaces.

\begin{proposition}\label{prop:existence}
 Assume that $(X,d)$ is a boundedly compact metric space with a locally finite measure $m$ and that 
 $\mu_0, \mu_1 \in \mathcal{P}(X)$ have bounded support. Then
 for all $C \ge 0$ and $\lambda \in (0,1)$ there exists a minimizer of $\mathcal{F}_C$ in $\mathcal{I}_\lambda(\mu_0,\mu_1)$.
\end{proposition}
\begin{proof}
 Take a sequence $(\nu_n)_{n=0}^\infty \subset \mathcal{I}_\lambda(\mu_0,\mu_1)$ so that
 \[
  \mathcal{F}_C(\nu_n) \to \inf\{\mathcal{F}_C(\omega) ~:~ \omega \in \mathcal{I}_\lambda(\mu_0,\mu_1)\}.
 \]
 Because by Lemma \ref{lma:cmpt} the set $\mathcal{I}_\lambda(\mu_0,\mu_1)$ is compact, we may assume that
 $\nu_n \to \nu \in \mathcal{I}_\lambda(\mu_0,\mu_1)$ in the $W_2$-distance.
 By Lemma \ref{lma:lsc}
 \[
  \mathcal{F}_C(\nu) \le \liminf_{n \to \infty}\mathcal{F}_C(\nu_n)
 \]
 and so we have the existence of the minimizer.
\end{proof}

\subsection{Existence of minimizers in $CD(K,\infty)$}

In the genuinely infinite dimensional case the set $\mathcal{I}_\lambda(\mu_0,\mu_1)$ does not have to be compact.
Therefore we will need to prove the existence of the needed minimizers by hand. Because we will need the existence of 
minimizers only for the set of midpoints, we will not formulate the results for other sets of intermediate points.

We will use the following lemma to prove the existence on minimizers.
The idea behind the lemma is very simple: we redistribute the possible excess mass using the assumption that we
are in a $CD(K,\infty)$ space and observe that the part of the redistributed measure
which has large density must necessarily be small. 

\begin{lemma}\label{lma:smallF}
 Assume that $(X,d)$ is a $CD(K,\infty)$ space and that 
 $\mu_0, \mu_1 \in \mathcal{P}^{ac}(X)$ with $\mu_0 = \rho_0 m$, $\mu_1 = \rho_1 m$ and
 $D = \diam(\spt \mu_0 \cup \spt \mu_1) < \infty$. Then
 for all 
 \[
  C \ge e^{K^-D^2/8} \max\{||\rho_0||_{L^{\infty}(X,m)},||\rho_1||_{L^{\infty}(X,m)}\}
 \]
 there exists $(H_\epsilon)_{\epsilon > 0} \subset \R$ with the following property.
 For each $\nu \in \mathcal{M}(\mu_0,\mu_1)$
 there exists $\tilde\nu \in \mathcal{M}(\mu_0,\mu_1)$ with
 \[
  \mathcal{F}_C(\tilde\nu) \le \mathcal{F}_C(\nu)
 \]
 and
 \[
  \mathcal{F}_{H_\epsilon}(\tilde\nu) \le \epsilon
 \] 
 for every $\epsilon > 0$.
\end{lemma}
\begin{proof}
 Take $x_0 \in X$ and $R>0$ so that the supports of all the measures in $\mathcal{M}(\mu_0,\mu_1)$ are
 contained in $B(x_0,R)$. By Proposition \ref{prop:doundedness} the measure $m$ is boundedly finite
 and so we have $m(B(x_0,R)) < \infty$.
 Take $\epsilon > 0$ and $C \ge M$, where
  \[
   M = e^{K^-D^2/8} \max\{||\rho_0||_{L^{\infty}(X,m)},||\rho_1||_{L^{\infty}(X,m)}\}.
  \]
 Let $\nu = \rho m + \nu^s \in \mathcal{M}(\mu_0,\mu_1)$ with $\nu^s \perp m$ and suppose that $\mathcal{F}_C(\nu) > 0$.
 Define a function $f \colon X \to [0,1]$ by
 \[
  f(x) = \begin{cases}
          1 - \frac{C}{\rho(x)}, & \text{if }\rho(x) \ge C \\
          0, & \text{if }\rho(x) < C.
         \end{cases}
 \]

 Let $\pi_1\in \GeoOpt(\nu, \mu_0)$ and $\pi_2\in \GeoOpt(\nu, \mu_1)$, and define $g \colon \Geo(X) \to [0,1]$ by
 \[
  g = (e_0)^{-1} \max\{f,\chi_A\},
 \]
 where $A \subset X$ is a Borel set with $m(A) = 0$ and $\nu^s(A) = \nu^s(X)$.
 Then
 \[
  (e_0)_\#(g\pi_1) = (e_0)_\#(g\pi_2) = f\rho m + \nu^s.
 \]
 Select a geodesic $\Gamma \in \Geo(\mathcal{P}(X))$ with 
 \[
  \Gamma_0 = \frac{(e_1)_\# (g\pi_1)}{\mathcal{F}_C(\nu)} \qquad\text{and}\qquad\Gamma_1 = \frac{(e_1)_\# (g\pi_2)}{\mathcal{F}_C(\nu)}
 \]
 so that the corresponding measure on geodesics satisfies \eqref{eq:CD-def2}. Then
 \[
  \sE_\infty(\Gamma_{\frac12}) \le \frac12\sE_\infty(\Gamma_0) + \frac12\sE_\infty(\Gamma_1) + \frac{K^-}2\frac12\left(1-\frac12\right)D^2
                                  \le \log \frac{M}{\mathcal{F}_C(\nu)}.
 \]
 On the other hand, writing $\Gamma_\frac12 = \rho_\frac12 m$,
 \begin{align*}
  \sE_\infty(\Gamma_{\frac12}) & = \int_{\{\rho_\frac12 \ge \delta\}} \rho_\frac12 \log \rho_\frac12 dm
                + \int_{\{0 \le \rho_\frac12 < \delta\}} \rho_\frac12 \log \rho_\frac12 dm \\
          & \ge \log \delta \int_{\{\rho_\frac12 \ge \delta\}} \rho_\frac12 dm - \frac{m(B(x_0,R))}{e}.
 \end{align*}
 Therefore with $\delta > 1$ we get
 \begin{equation}\label{eq:Fbound}
    \mathcal{F}_\delta (\Gamma_{\frac12}) \le \int_{\{\rho_\frac12 \ge \delta\}} \rho_\frac12 dm 
      \le \frac{1}{\log \delta}\left(\log\frac{M}{\mathcal{F}_C(\nu)} + \frac{m(B(x_0,R))}{e}\right).
 \end{equation} 

 Define
 \[
  \omega = (1-f)\rho m + \mathcal{F}_C(\nu)\Gamma_{\frac12}.
 \]
 By Lemma \ref{lma:combined} we have $\omega \in \mathcal{M}(\mu_0,\mu_1)$. By taking $H_\epsilon>C$ so large that
 \[
  \frac{1}{\log H_\epsilon}\left(\log\frac{M}{\epsilon} + \frac{m(B(x_0,R))}{e}\right) \le \epsilon
 \]
  we get from \eqref{eq:Fbound} the required estimate
  \[
   \mathcal{F}_{H_\epsilon}(\omega) \le \mathcal{F}_C(\nu)\mathcal{F}_{H_\epsilon}(\Gamma_{\frac12}) \le \epsilon
  \]
 which proves the claim.
\end{proof}

In the boundedly compact case we were able to prove the existence of the minimizers of $\mathcal{F}_C$
for all values of $C$. In $CD(K,\infty)$ spaces we get the existence only for the values
that are greater than or equal to a critical threshold. Fortunately these are the only values of $C$
that will be needed in the proof for the existence of a good geodesic.

\begin{proposition}\label{prop:existence_noncompact}
 Assume that $(X,d)$ is a $CD(K,\infty)$ space and that 
 $\mu_0, \mu_1 \in \mathcal{P}^{ac}(X)$ with $\mu_0 = \rho_0 m$, $\mu_1 = \rho_1 m$ and
 $D = \diam(\spt \mu_0 \cup \spt \mu_1) < \infty$. Then for all
 \[
  C \ge e^{K^-D^2/8} \max\{||\rho_0||_{L^{\infty}(X,m)},||\rho_1||_{L^{\infty}(X,m)}\}
 \]
 there exists a minimizer of $\mathcal{F}_C$ in $\mathcal{M}(\mu_0,\mu_1)$.
\end{proposition}
\begin{proof}
 Take a sequence $(\nu_n)_{n=0}^\infty \subset \mathcal{M}(\mu_0,\mu_1)$ so that
 \[
  \mathcal{F}_C(\nu_n) \to \inf\{\mathcal{F}_C(\omega) ~:~ \omega \in \mathcal{M}(\mu_0,\mu_1)\}.
 \]
 By Lemma \ref{lma:smallF} there exists a sequence $(H_k)_{k=0}^\infty \subset [0,\infty)$ so that,
 by redefining the sequence $(\nu_n)_{n=0}^\infty$ if necessary, we may assume for all
 $n \ge k \ge 0$ the estimate 
 \begin{equation}\label{eq:sas}
  \mathcal{F}_{H_k}(\nu_n) \le 2^{-k}.
 \end{equation}
 Because $D < \infty$ we have
 \[
  \mathcal{M}(\mu_0,\mu_1) \subset \{\omega \in \mathcal{P}(X) ~:~ \spt \omega \subset B\}
 \]
 for some closed and bounded set $B \subset X$.
 By Proposition \ref{prop:doundedness} the measure $m$ is boundedly finite and so the set 
 \[
  \mathcal{A}_H = \{\omega \in \mathcal{P}(X) ~:~ \mathcal{F}_H(\omega) = 0 \text{ and }\spt \omega \subset B\}
 \]
 is relatively compact in $(\mathcal{P}(X),W_2)$ and nonempty for all $H \ge C$. On the other hand, by \eqref{eq:sas}
 we have
 \[
  W_2(\nu_n, \mathcal{A}_{H_k}) \le 2^{-k}D
 \]
 for all $n \ge k \ge 0$.
 Using this with $k=1$ gives the existence of a subsequence $(\nu_{1_n})_{n=0}^\infty$ of $(\nu_n)_{n=0}^\infty$ with
 \[
  W_2(\nu_{1_i},\nu_{1_j}) \le D
 \]
 for all $i,j \in \N$. Inductively using \eqref{eq:sas} we define for all $k \ge 1$ a subsequence
 $(\nu_{k_n})_{n=0}^\infty$ of $(\nu_{(k-1)_n})_{n=0}^\infty$ so that
 \[
  W_2(\nu_{k_i},\nu_{k_j}) \le 2^{1-k}D
 \]
 for all $i,j \in \N$. By a diagonal argument we then get a subsequence converging in the Wasserstein distance
 to a measure $\nu$ which is in $\mathcal{M}(\mu_0,\mu_1)$ by Lemma \ref{lma:bdd}. Then by Lemma \ref{lma:lsc} 
 we conclude that the measure $\nu$ is a minimizer of $\mathcal{F}_C$ in $\mathcal{M}(\mu_0,\mu_1)$.
\end{proof}

\begin{remark}
 Notice that if we knew a priori that
 \begin{equation}\label{eq:inf}
  \inf_{\omega \in \mathcal{M}(\mu_0,\mu_1)}\mathcal{F}_C(\omega) = 0,  
 \end{equation}
 then the existence of the minimizer in Proposition \ref{prop:existence_noncompact} would follow immediately 
 without Lemma \ref{lma:smallF}. However, our proof for \eqref{eq:inf} in Proposition \ref{prop:middensity} 
 will use the existence of the minimizer, so Lemma \ref{lma:smallF} here seems to be a necessary step.
\end{remark}

\subsection{$L^\infty$-estimate for the minimizers}

Now that we have established the needed basic properties of the set $\mathcal{I}_\lambda(\mu_0,\mu_1)$ and
the functional $\mathcal{F}_C$
we turn to the properties of the minimizers. What we are aiming at here is an $L^\infty$-bound on the density of a good midpoint.
In order to quantify some estimates in the proof we first have to go slightly above the final threshold.

\begin{proposition}\label{prop:middensity}
 Assume that $(X,d,m)$ is a $CD(K,N)$ space for some $K \in \R$ and $N \in (0,\infty]$
 and that $\mu_0, \mu_1 \in \mathcal{P}^{ac}(X,m)$
 have bounded support and densities $\rho_0$ and $\rho_1$, respectively. Suppose in addition that all
 measures in $\GeoOpt(\mu_0,\mu_1)$ are concentrated on geodesics with length at most $D$. Then for any
 \[
  C > C(N,K,D)\max\{||\rho_0||_{L^{\infty}(X,m)},||\rho_1||_{L^{\infty}(X,m)}\}
 \]
 we have
 \[
  \min_{\nu \in \mathcal{M}(\mu_0,\mu_1)} \mathcal{F}_C(\nu) = 0.
 \]
\end{proposition}

\begin{proof}
 Write
 \[
  M = C(N,K,D)\max\{||\rho_0||_{L^{\infty}(X,m)},||\rho_1||_{L^{\infty}(X,m)}\}.
 \]
 Suppose that the conclusion is not true. Let $\mathcal{M}_\text{min} \subset \mathcal{M}(\mu_0,\mu_1)$ be the set of
 minimizers of $\mathcal{F}_C$ in $\mathcal{M}(\mu_0,\mu_1)$, which by Proposition \ref{prop:existence} and
 Proposition \ref{prop:existence_noncompact} is always nonempty.
 Take $\nu \in \mathcal{M}_\text{min}$  for which
 \begin{equation}\label{eq:almostmaxpos}
  m(\{x \in X ~:~ \rho_\nu(x) > C\}) \ge \left(\frac{M}{C}\right)^{\frac14} \sup_{\omega \in \mathcal{M}_\text{min}}m(\{x \in X ~:~ \rho_\omega(x) > C\}),
 \end{equation}
 where $\nu = \rho_\nu m + \nu^s$ with $\nu^s \perp m$ and $\omega = \rho_\omega m + \omega^s$ with $\omega^s \perp m$.

 Assume first that the set 
 \[
  A = \{x \in X ~:~ \rho_\nu(x) > C\}  
 \]
 has positive $m$-measure.
 Then there exists $\delta >0$ so that 
 \[
  m(A') > \left(\frac{M}{C}\right)^{\frac12} m(A)  
 \]
 with
 \begin{equation}\label{eq:deltabound}
  A' = \{x \in A ~:~ \rho_\nu(x) > C + \delta\}.  
 \end{equation}

 Let $\pi_1\in \GeoOpt(\nu, \mu_0)$ and $\pi_2\in \GeoOpt(\nu, \mu_1)$, and take a geodesic $\Gamma \in \Geo(\mathcal{P}(X))$ 
 given by Proposition \ref{prop:spreading} with 
 \[
  \Gamma_0 = \frac{(e_1)_\# \pi_1|_{\{\gamma_0 \in A'\}}}{\nu(A')} \qquad\text{and}\qquad\Gamma_1 = \frac{(e_1)_\# \pi_2|_{\{\gamma_0 \in A'\}}}{\nu(A')}
 \]
 such that the corresponding measure on geodesics satisfies \eqref{eq:bigsupport}.

 We write $\Gamma_{\frac12} = \rho_\Gamma m + \Gamma^s$ with $\Gamma^s \perp m$ and abbreviate
 \[
  E = \left\{x \in X ~:~ \rho_\Gamma(x) > 0\right\}.  
 \]
 From \eqref{eq:bigsupport} we get 
 \[
  m(E) \ge \frac{\nu(A')}{M} \ge \frac{C}{M} m(A') \ge \left(\frac{C}{M}\right)^{\frac12} m(A).
 \]

 Now consider a new measure $\tilde\nu = \rho_{\tilde\nu} m + \tilde\nu^s$, with $\tilde\nu^s \perp m$, defined as the combination
 \[
  \tilde\nu = \nu|_{X \setminus A'} + \frac{C}{C+ \delta} \nu|_{A'} + \frac{\delta}{C+\delta}\nu(A') \Gamma_{\frac12}.
 \]

 By Lemma \ref{lma:combined} we have $\tilde\nu \in \mathcal{M}(\mu_0,\mu_1)$. Due to the definition \eqref{eq:deltabound}
 we only redistribute some of the mass above the density $C$ when we replace the measure $\nu$ by the measure $\tilde\nu$.
 See Figure \ref{fig:redistribution} for an illustration of the redistributed part of the measure.
 Let us now calculate how much the excess mass functional changes in this replacement.
 \begin{align*}
  \mathcal{F}_C(\nu) &- \mathcal{F}_C(\tilde\nu) 
    =  \int_X \left(\rho_\nu - C\right)^+dm + \nu^s(X) - \int_X \left(\rho_{\tilde\nu} - C\right)^+dm - \tilde\nu^s(X) \\
    = & \int_{X \setminus A'}\left(\left(\rho_\nu - C\right)^+ - \left(\rho_\nu + \frac{\delta}{C+\delta}\nu(A')\rho_\Gamma - C\right)^+ \right)dm\\
    & + \int_{A'}\left(\left(\rho_\nu - C\right)^+ - \left(\frac{C}{C+\delta}\rho_\nu + \frac{\delta}{C+\delta}\nu(A')\rho_\Gamma - C\right)^+ \right)dm\\
    & + \frac{\delta}{C+\delta}\left(\nu^s(A') - \nu(A')\Gamma^s(X)\right)\\
    = & \int_{X \setminus A'}\left(\left(\rho_\nu - C\right)^+ - \left(\rho_\nu + \frac{\delta}{C+\delta}\nu(A')\rho_\Gamma - C\right)^+ \right)dm\\
    & + \int_{A'}\frac{\delta}{C+\delta}\left(\rho_\nu - \nu(A')\rho_\Gamma\right)dm + \frac{\delta}{C+\delta}\left(\nu^s(A') - \nu(A')\Gamma^s(X)\right)\\
    = & \int_{X \setminus A'}\left(\left(\rho_\nu - C\right)^+ - \left(\rho_\nu + \frac{\delta}{C+\delta}\nu(A')\rho_\Gamma - C\right)^+ + \frac{\delta}{C+\delta}\nu(A')\rho_\Gamma\right)dm\\   
   =  & \int_{\{\rho_\nu < C \le \frac{\delta}{C+\delta}\nu(A')\rho_\Gamma + \rho_\nu\}}(C-\rho_\nu)dm + \int_{\{\frac{\delta}{C+\delta}\nu(A')\rho_\Gamma + \rho_\nu < C\}}\frac{\delta}{C+\delta}\nu(A')\rho_\Gamma dm\\
   = & \int_{\{\rho_\nu < C\}}\min\left\{C-\rho_\nu, \frac{\delta}{C+\delta}\nu(A')\rho_\Gamma\right\}dm.
 \end{align*}

\begin{figure}
  \centering
  \includegraphics[width=\textwidth]{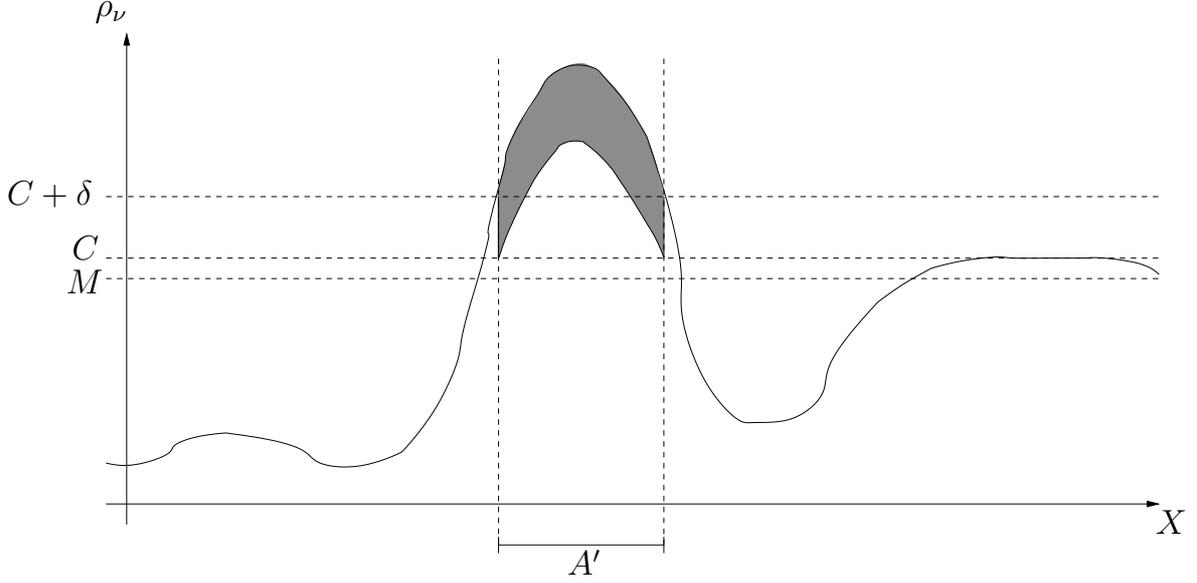}
  \caption{When we replace the measure $\nu$ by the new measure $\tilde\nu$ we redistribute the gray part of the measure.}
  \label{fig:redistribution}
 \end{figure}
 

 Because of the minimality of $\mathcal{F}_C$ at $\nu$ this integral must be zero. Therefore also
 \[
  m(E \cap \{x \in X ~:~ \rho_\nu(x) < C\}) = 0.
 \]
 On the other hand, for any $y \in E \cap \{x \in X : \rho_\nu(x) \ge C\}$ we have $\rho_{\tilde\nu}(y) > C$.
 This together with the assumption \eqref{eq:almostmaxpos} leads to a contradiction
 \begin{align*}
  m(\{x \in X ~:~ \rho_{\tilde\nu}(x) > C\}) & \ge m(E) \ge \left(\frac{C}{M}\right)^{\frac12} m(A) \\
    & \ge \left(\frac{C}{M}\right)^{\frac14} \sup_{\omega \in \mathcal{M}_\text{min}}m(\{x \in X ~:~ \rho_\omega(x) > C\}).
 \end{align*}

 Suppose now that $m(A)=0$. This means that $\nu$
 must have a singular part. Similarly as above, we can redistribute this singular part using \eqref{eq:bigsupport}. This
 leads immediately to a contradiction because at the combination of the redistributed singular part and the absolutely continuous
 part of $\nu$ the functional $\mathcal{F}_C$ has lower value than at $\nu$.
\end{proof}

Now we can obtain the correct threshold level using the previous Proposition \ref{prop:middensity}.

\begin{corollary}\label{cor:middensity}
 With the assumptions of Proposition \ref{prop:middensity} there exists
 $\nu \in \mathcal{M}(\mu_0,\mu_1)$ with $\mathcal{F}_C(\nu) = 0$ for
 $
  C = C(N,K,D)\max\{||\rho_0||_{L^{\infty}(X,m)},||\rho_1||_{L^{\infty}(X,m)}\}.
 $
\end{corollary}
\begin{proof}
 By Proposition \ref{prop:middensity} we know that 
 \[
  \min_{\omega \in \mathcal{M}(\mu_0,\mu_1)}\mathcal{F}_{C'}(\omega) = 0
 \]
 for all $C'>C$. Because $\mu_0$ and $\mu_1$ have bounded support, all the measures in $\mathcal{M}(\mu_0,\mu_1)$ are supported
 on a bounded set $A \subset X$. Therefore,
 \[
  \min_{\omega \in \mathcal{M}(\mu_0,\mu_1)}\mathcal{F}_{C}(\omega) 
   \le \min_{\omega \in \mathcal{M}(\mu_0,\mu_1)}\mathcal{F}_{C'}(\omega) + (C' - C)m(A) \to 0
 \]
 as $C' \searrow C$.
\end{proof}

\subsection{From the midpoints to a geodesic}

Corollary \ref{cor:middensity} together with Lemma \ref{lma:lsc} now gives the geodesic $\Gamma$ of 
Theorem \ref{thm:goodgeodesic}.

\begin{proof}[Proof of Theorem \ref{thm:goodgeodesic}]
 Let us first define the geodesic $\Gamma$ for a dense set of parameters in the following inductive manner:
 first set $\Gamma_0 = \mu_0$ and $\Gamma_1 = \mu_1$. Now assume that for some $n \in \N$ we have defined
 $\Gamma_{k2^{-n}} = \rho_{k2^{-n}} m$ for all integers $0 \le k \le 2^{n}$ and that for these we have
 \begin{equation}\label{eq:indassum}
  ||\rho_{k2^{-n}}||_{L^{\infty}(X,m)} \le 
      \prod_{i=1}^nC(N,K,2^{-i+1}D)\max\{||\rho_0||_{L^{\infty}(X,m)},||\rho_1||_{L^{\infty}(X,m)}\}.
 \end{equation}
 
 Because of the assumption $D < \infty$
 and the fact that any geodesic in the Wasserstein space $(\mathcal{P}(X),W_2)$ between $\mu_0$ and $\mu_1$ can
 be considered as a measure in $\GeoOpt(\mu_0,\mu_1)$, we have that any measure in $\GeoOpt(\Gamma_{k2^{-n}},\Gamma_{(k+1)2^{-n}})$
 is concentrated on geodesics with length at most $2^{-n}D$.

 Now define for all odd $0 \le k \le 2^{n+1}$ the measure $\Gamma_{k2^{-n-1}} = \rho_{k2^{-n-1}}m$ to be a measure in
 $\mathcal{M}(\Gamma_{(k-1)2^{-n-1}},\Gamma_{(k+1)2^{-n-1}})$ given by Corollary \ref{cor:middensity}.
 Then by our inductive assumption \eqref{eq:indassum} the estimate
 \begin{align*}
  ||\rho_{k2^{-n-1}}||_{L^{\infty}(X,m)} 
     & \le C(N,K,2^{-n}D)\max\{||\rho_{(k-1)2^{-n-1}}||_{L^{\infty}(X,m)},||\rho_{(k+1)2^{-n-1}}||_{L^{\infty}(X,m)}\}\\
     & \le \prod_{i=1}^{n+1}C(N,K,2^{-i+1}D)\max\{||\rho_0||_{L^{\infty}(X,m)},||\rho_1||_{L^{\infty}(X,m)}\}      
 \end{align*} 
 holds. The rest of the geodesic $\Gamma$ is defined by completion. The validity of the estimates \eqref{eq:densitybound_finite}
 and \eqref{eq:densitybound_infinite} for all $t \in [0,1]$ follow then from Lemma \ref{lma:lsc}.
\end{proof}

\section{Local Poincar\'e inequalities using the good geodesics}\label{sec:poincare}

Let us now show how the density bounds we have obtained imply the local Poincar\'e inequalities.
Although this part of the proof is almost the same as the one given in \cite{R2011} for the Poincar\'e inequalities in metric spaces
with Ricci curvature bounded from below in the sense of Lott and Villani, we will repeat the proof for the convenience
of the reader. Notice also that the proof we follow from \cite{R2011} for a large part follows the proof of
 \cite[Theorem 2.5]{LV2007}. 

The difference here to the proof in \cite{R2011} is that we have chosen to define the sets $B^+$ and $B^-$
slightly differently so that the proof works also for measures $m$ that have atoms. This change results in an extra multiplication
by $2$ of the constant in the Poincar\'e inequality. Since already the constant given by the proof in \cite{R2011} was 
not sharp, we do not care too much about increasing the constant slightly in order to simplify the exposition.

\begin{theorem}\label{thm:general}
 Let $(X,d)$ be a metric space with a boundedly finite measure $m$. 
 Suppose that there exists a function $C \colon [0,\infty) \to [1,\infty)$ so that for any
 $\mu_0, \mu_1 \in \mathcal{P}^{ac}(X,m)$ with $D = \diam(\spt\mu_0\cup\spt \mu_1) < \infty$ there
 exists a measure $\pi \in \GeoOpt(\mu_0,\mu_1)$ so that for all $t \in [0,1]$ we have $(e_t)_\#\pi = \rho_tm$ with
 \begin{equation}\label{eq:densitybound_general}
  ||\rho_t||_{L^{\infty}(X,m)} \le C(D) \max\{||\rho_0||_{L^{\infty}(X,m)},||\rho_1||_{L^{\infty}(X,m)}\}.
 \end{equation}
 Then the space $(X,d,m)$ supports the local Poincar\'e type inequality
 \[
  \int_{B(x,r)}|u - \langle u\rangle_{B(x,r)}|dm \le 8 r C(2r) \int_{B(x,2 r)}gdm.
 \]
\end{theorem}

\begin{proof}
 Abbreviate $B = B(x,r)$ and define $M$ to be the median of $u$ in the ball $B$, i.e.
 \[
   M = \inf\left\{a \in \R : m(\{u > a\}) \le \frac{m(B)}2\right\}.
 \]
 Using the median $M$ we cover the ball $B$ with two Borel sets
 \[
  B^+ = \{x \in B ~:~ u(x) \ge M\}\qquad \text{and}\qquad B^- = \{x \in B ~:~ u(x) \le M\}.
 \]
 Notice that $m(B^+), m(B^-) \ge m(B)/2$. Let
 \[
  \pi \in \GeoOpt\left(\frac1{m(B^+)}m|_{B^+},\frac1{m(B^-)}m|_{B^-}\right)
 \]
 be the geodesic given \eqref{eq:densitybound_general} and let $\rho_t$ be the density of $(e_t)_\#\pi$ with respect to $m$. 
 By \eqref{eq:densitybound_general} we have for all $t \in [0,1]$ at $m$-almost every $y\in X$
 \[
  \rho_t(y) \le C(2r)\frac{2}{m(B)}.
 \]

 Now observe that we have an equality
 \[
   |u(\gamma_0) - u(\gamma_1)| = |u(\gamma_0) - M| + |M - u(\gamma_1)|
 \]
 for $\pi$-almost every $\gamma \in \Geo(X)$. Therefore
 \begin{align*}
  \int_{\Geo(X)} & |u(\gamma_0) -  u(\gamma_1)| d\pi(\gamma)\\
   & = \int_{\Geo(X)} |u(\gamma_0) - M|d\pi(\gamma) + \int_{\Geo(X)}|M - u(\gamma_1)|d\pi(\gamma)\\
   & = \frac{1}{m(B^+)}\int_{B^+}|u(x) - M|dm(x) + \frac{1}{m(B^-)}\int_{B^-}|M-u(x)|dm(x)\\
   & \ge \frac{1}{m(B)}\int_{B}|u(x) - M|dm(x).
 \end{align*}
 Since $\pi$-almost every $\gamma \in \Geo(X)$ is contained in the ball $B(x,2r)$ we have
{\allowdisplaybreaks 
 \begin{align*}
  \int_{B(x,r)}&|u - \langle u\rangle_{B(x,r)}|dm  \le \frac{1}{m(B)}\iint_{B \times B} |u(x) - u(y)|dm(x)dm(y) \\
   & \le \frac{1}{m(B)}\iint_{B \times B} (|u(x) - M| + |M - u(y)|)dm(x)dm(y) \\
   & = 2\int_{B} |u(x) - M|dm(x) 
    \le 2m(B)\int_{\Geo(X)} |u(\gamma_0) - u(\gamma_1)| d\pi(\gamma)\\
   & \le 4rm(B) \int_{\Geo(X)}\int_0^1 g(\gamma_t)dt d\pi(\gamma)
    = 4rm(B) \int_0^1 \int_Xg(x)\rho_t(x)dm(x)dt\\
   & \le 8r C(2r) \int_0^1 \int_{B(x,2 r)}g(x)dm(x)dt 
    = 8 r C(2r) \int_{B(x,2 r)}gdm.
 \end{align*}
 }
\end{proof}

 Theorem \ref{thm:main2} now follows immediately by combining Theorem \ref{thm:goodgeodesic} and Theorem \ref{thm:general}.
 To get Theorem \ref{thm:main} we have to recall also the Proposition \ref{prop:doubling}.

 Let us end this section by noting that the existence of good geodesics and hence the local Poincar\'e inequality
 follows also from the assumption that we have displacement convexity for some functional from quite a large class of functionals.
 Let $F \colon [0,\infty) \to \R$  be a convex function. From it we define a functional 
 $\sF \colon \mathcal{P}(X) \to [-\infty,\infty]$ by setting
 \begin{equation}\label{eq:functionaldef}
   \sF(\mu) = \int_X F(\rho)dm + F'(\infty)\mu^s(X),  
 \end{equation}
 where $\mu = \rho m + \mu^s$, $\mu^s \perp m$ and the derivative at infinity is defined as
 \[
  F'(\infty) = \lim_{r \to \infty}\frac{F(r)}{r}.
 \]
 We say this functional is
 weakly displacement convex in the space $(\mathcal{P}(X),W_2)$ if for any two measures $\mu_0, \mu_1 \in \mathcal{P}(X)$ with $W_2(\mu_0,\mu_1) < \infty$
 there exists a measure $\pi \in \GeoOpt(\mu_0,\mu_1)$ so that
 \[
  \sF((e_t)_\#\pi) \le (1-t)\sF(\mu_0) + t\sF(\mu_1).
 \]

 \begin{theorem}\label{thm:convexitypoincare}
  Let $(X,d)$ be boundedly compact metric spaces with a locally finite measure $m$ and $F \colon [0,\infty) \to \R$
  a convex function for which $F(x)/x$ is strictly increasing, $F(0)=0$ and $F'(\infty) = \infty$. Suppose that the corresponding
  functional $\sF$ given by \eqref{eq:functionaldef} is weakly displacement convex in $(\mathcal{P}(X),W_2)$.

  Then for any $\mu_0, \mu_1 \in \mathcal{P}^{ac}(X,m)$ with $D = \diam(\spt\mu_0\cup\spt \mu_1) < \infty$ there
  exists a measure $\pi \in \GeoOpt(\mu_0,\mu_1)$ so that for all $t \in [0,1]$ we have $(e_t)_\#\pi = \rho_tm$ with
 \[
  ||\rho_t||_{L^{\infty}(X,m)} \le \max\{||\rho_0||_{L^{\infty}(X,m)},||\rho_1||_{L^{\infty}(X,m)}\}
 \]
  In particular, we have the local Poincar\'e type inequality
  \[
   \int_{B(x,r)}|u - \langle u\rangle_{B(x,r)}|dm \le 8 r \int_{B(x,2 r)}gdm.
  \]
 \end{theorem}
 \begin{proof}
  The local Poincar\'e type inequality follows from the density bound via Theorem \ref{thm:general}.
  Therefore we only have to prove the density bound.
  Take $\mu_0, \mu_1 \in \mathcal{P}^{ac}(X,m)$ with bounded support and with densities $\rho_0$ and $\rho_1$
  bounded from above and let $\pi \in \GeoOpt(\mu_0,\mu_1)$ be a measure along which we have displacement convexity.
  Write
  \[
   M = \max\{||\rho_0||_{L^{\infty}(X,m)},||\rho_1||_{L^{\infty}(X,m)}\}
  \]
  and
  \[
   E = \{x \in X ~:~ \rho_{\frac12}(x) > 0\}.
  \]
  Now from the weak displacement convexity we get
  \begin{align*}
   \sF((e_\frac12)_\#\pi) & \le \frac12 \sF(\mu_0) + \frac12 \sF(\mu_1) = \frac12 \int_X F(\rho_0)dm + \frac12 \int_X F(\rho_1)dm\\
                          & = \frac12 \int_X \frac{F(\rho_0)}{\rho_0}\rho_0dm + \frac12 \int_X \frac{F(\rho_1)}{\rho_1}\rho_1dm \\
                          & \le \frac12 \int_X \frac{F(M)}{M}\rho_0dm + \frac12 \int_X \frac{F(M)}{M}\rho_1dm = \frac{F(M)}{M}.
  \end{align*}
  In particular $(e_\frac12)_\#\pi$ has no singular part and then by Jensen's inequality
  \begin{align*}
   \sF((e_\frac12)_\#\pi) & = \int_E F(\rho_\frac12)dm = m(E) \dashint_E F(\rho_\frac12)dm \\
                          & \ge  m(E) F\left(\dashint_E \rho_\frac12dm \right) = m(E)F(m(E)^{-1}).
  \end{align*}
  Combining these two estimates with the fact that $F(x)/x$ is strictly increasing yields
  \[
   m(E) \ge \frac1M.
  \]
  Thus the considerations of Section \ref{sec:construction} work also in this situation and the density bound follows.
 \end{proof}

\section{$MCP(K,N)$ property on $CD(K,N)$ spaces}\label{sec:mcp}

In this section we construct another set of good geodesics in $CD(K,N)$ spaces (where $N < \infty$) with sharp density bounds
using the minimizing procedure of Section \ref{sec:construction}. These geodesics are constructed between a point mass
and a uniformly distributed measure. Such geodesics are the ones that are used
in the definition of the measure contraction property $MCP(K,N)$. So, once we have found these geodesics we have 
proved the $MCP(K,N)$ property.

Construction of the needed geodesics relies on the same techniques that were used in Section \ref{sec:construction}. Instead of minimizing 
$\mathcal{F}_C$ among midpoints between the measures $\mu_0$ and $\mu_1$, we will take a $\lambda \in(0,1)$ and minimize $\mathcal{F}_C$
in $\mathcal{I}_\lambda(\mu_0,\mu_1)$. This minimization together
with the lower semicontinuity of $\mathcal{F}_C$ gives us the needed bounds already for a sequence of intermediate measures, as will
be seen in Lemma \ref{lma:mcp_sharpmid}.

\begin{remark}\label{rmk:annulardisjoint}
 In verifying the measure contraction property we will consider geodesics between measures $\mu_0 = \delta_x$ and
 $\mu_1 = \frac1{m(A)}m|_A$. Because the restrictions of the measure $\mu_1$ to annular regions
 \[
  A_k = B(x,r^k) \setminus B(x,r^{k-1}), \qquad{k \in \Z}
 \]
 have pairwise disjoint supports even when we move them along any geodesic towards $\mu_0$, we can define the intermediate measures
 and the geodesic separately for each such annular region. This for example allows as to make the assumption that $A$ is bounded.
\end{remark}

 In the following lemma we will use the notation $A_k$ of previous remark and also abbreviate a dilated annulus by
 \[
  sA_k = B(x,tr^k) \setminus B(x,sr^{k-1})
 \]
 for all $s \in [0,1]$.

\begin{lemma}\label{lma:mcp_sharpmid}
 Let $x \in X$ and $A \subset X$ with $0 < m(A) < \infty$. Suppose that we have $\pi \in \GeoOpt(\mu_0,\mu_1)$ with
 $\mu_0 = \delta_x$ and $\mu_1 = \frac1{m(A)}m|_A$ and $t \in (0,1]$ for which we have
 \begin{equation}\label{eq:MCPdef_revisit}
  dm \ge (e_{t})_\#\left(t^N\beta_{t}(\gamma_0,\gamma_1)m(A)d\pi(\gamma)\right).
 \end{equation}
 Then for any $\lambda \in (0,1)$ there exists $\tilde\pi \in \GeoOpt(\mu_0,\mu_1)$ so that 
 \[
  (e_{s})_\#\pi = (e_{s})_\#{\tilde\pi}
 \]
 for all $s \in [t,1]$ and \eqref{eq:MCPdef_revisit} holds also with $t$ replaced by $\lambda t$ and 
 $\pi$ replaced by $\tilde\pi$.
\end{lemma}

\begin{proof}
 Take $r>1$. With Remark \ref{rmk:annulardisjoint} in mind we can define the intermediate measure separately
 for different annuli. Take $k \in \Z$ so that $m(A_k)> 0$ where $A_k$ is an annulus as in Remark \ref{rmk:annulardisjoint}.
 By \eqref{eq:MCPdef_revisit} we have for the density $\rho$ of $(e_{t})_\#\pi$ with respect to $m$ the estimate
 \[
  \rho \le \frac1{t^{N}\min\{\beta_{t}(r^{k}),\beta_{t}(r^{k-1})\}m(A)}\qquad \text{for all }y \in tA_k.
 \]
 
 Now any
 \[
  \pi_k \in \GeoOpt\left(\mu_0, \frac{m(A)}{m(A_k)}((e_{t})_\#\pi)|_{tA_k}\right)  
 \]
 is concentrated on geodesics with length between $tr^{k-1}$ and $tr^{k}$.
 
 Therefore by \eqref{eq:CD-def} there exists a measure $\pi_k$ with
 \begin{align*}
  \sE_N\left((e_{\lambda})_\#\pi_k\right) & \le - \lambda \left(t^{N}\min\{\beta_{\lambda}(tr^{k}),\beta_{\lambda}(tr^{k-1})\}\min\{\beta_{t}(r^{k}),\beta_{t}(r^{k-1})\}m(A_k)\right)^\frac1N\\
   & = - \lambda t\left(\min\{\beta_{\lambda t}(r^{k}),\beta_{\lambda t}(tr^{k-1})\}m(A_k)\right)^\frac1N.
 \end{align*}
 
 Then with the help of Jensen's inequality as in the proof of Proposition \ref{prop:spreading} and
 with a similar proof as for Proposition \ref{prop:middensity} we get
 a good intermediate measure 
 \[
  \nu_k \in \mathcal{I}_{1-\lambda}\left(\mu_0,\frac{m(A)}{m(A_k)}((e_{t})_\#\pi)|_{tA_k}\right)
 \]
 which has the density $\rho_k$ with respect to $m$ with the bound
 \begin{align*}
  \rho_k & \le \frac{1}{(\lambda t)^N \min\{\beta_{\lambda t}(r^{k}),\beta_{\lambda t}(r^{k-1})\}m(A_k)}.
 \end{align*}

 Now the sum
 \[
  \nu = \sum_{k \in \Z} m(A_k)\nu_k
 \]
 has the correct density bound locally up to a constant which tends to one as $r \searrow 1$. Hence by Lemma \ref{lma:cmpt}
 we find a sequence converging to a measure where we have the correct density bound by Lemma \ref{lma:lsc}. This measure
 induces the desired $\tilde\pi$.
\end{proof}
 
The proof now follows using the lower semicontinuity of $\mathcal{F}_C$ and the compactness of the set of geodesics
between $\mu_0$ and $\mu_1$.

\begin{proof}[Proof of Theorem \ref{thm:mcp}]
 Let $x \in X$ and $A \subset X$ with $0 < m(A) < \infty$. Because of the Remark \ref{rmk:annulardisjoint}
 we may assume that $A$ is bounded. Write $\mu_0 = \delta_x$ and $\mu_1 = \frac1{m(A)}m|_A$.
 By Lemma \ref{lma:mcp_sharpmid} we get for every $n \in \N$ a measure $\pi_n \in \GeoOpt(\mu_0,\mu_1)$ with
 \begin{equation}\label{eq:MCPdef_revisit2}
   dm \ge (e_{t})_\#\left(t^N\beta_{t}(\gamma_0,\gamma_1)m(A)d\pi_n(\gamma)\right).
 \end{equation}
 for all $t = \frac{k}{2^n}$, with $k = 1, 2, \dots, 2^n$.

 By Lemma \ref{lma:cmpt_geod} the sequence $(\pi_n)_{n=1}^\infty$ has a converging subsequence in the $\mathcal{W}_2$-distance.
 From Lemma \ref{lma:lsc} we see that the limit $\pi \in \GeoOpt(\mu_0,\mu_1)$ of this subsequence then satisfies 
 \eqref{eq:MCPdef_revisit2} for all $t \in [0,1]$.
\end{proof}

\end{document}